\documentclass{article}

\usepackage[all,ps,arc]{xy}  
\usepackage{amsmath,amssymb,latexsym}
\usepackage{amsfonts}
\usepackage{amsthm}
\usepackage{xcolor}

\newcommand{\point}{\ensuremath{\xymatrix{A\ar@<+.6ex>[r]^(.5){\alpha}
&B\ar@<+.6ex>[l]^(.5){\beta}}}}
\newcommand{\rg}{\ensuremath{\xymatrix{A\ar@<+1ex>[r]^{\alpha}\ar@<-1ex>[r]_{\gamma}&B\ar[l]|{\beta}}}}

\newtheorem{Theorem}{Theorem}[section]
\newtheorem{Lemma}[Theorem]{Lemma}
\newtheorem{Proposition}[Theorem]{Proposition}
\newtheorem{Definition}[Theorem]{Definition}

\newtheorem{Corollary}[Theorem]{Corollary}

\theoremstyle{remark}\newtheorem{Remark}[Theorem]{Remark}
\newtheorem{Example}[Theorem]{Example}

\newcommand{\bA}{\ensuremath{\mathbb A}}
\newcommand{\bB}{\ensuremath{\mathbb B}}

\newcommand{\bI}{\ensuremath{\mathbb I}}

\newcommand{\cA}{\ensuremath{\mathsf A}}
\newcommand{\cB}{\ensuremath{\mathsf B}}
\newcommand{\cC}{\ensuremath{\mathsf C}}

\newcommand{\cI}{\ensuremath{\mathsf I}}

\newcommand{\cX}{\ensuremath{\mathsf X}}

\newcommand\ACT{\ensuremath{\mathsf{ACT}}}

\newcommand\CG{\ensuremath{2\mathbb{G}\mathsf{p}}}
\newcommand\SCG{\ensuremath{\mathbb{S}\mathsf{ym}2\mathbb{G}\mathsf{p}}}
\newcommand\SMC{\ensuremath{\mathbb{S}\mathsf{ymMonCat}}}

\newcommand\Set{\ensuremath{\mathsf{Set}}}
\newcommand\Gp{\ensuremath{\mathsf{Gp}}}

\newcommand\OPEXT{\ensuremath{\mathsf{OPEXT}}}
\newcommand\ABEXT{\ensuremath{\mathsf{ABEXT}}}

\newcommand\Mod{\ensuremath{\mathsf{Mod}}}
\newcommand\Mon{\ensuremath{\mathsf{Mon}}}
\newcommand\CMon{\ensuremath{\mathsf{CMon}}}
\newcommand\Cat{\ensuremath{\mathsf{Cat}}}
\newcommand\CAT{\ensuremath{\mathbb{C}\mathsf{at}}}
\newcommand\Ab{\ensuremath{\mathsf{Ab}}}

\newcommand{\opFib}{\ensuremath{\mathsf{op}\mathbb{F}\mathsf{ib}}}

\newcommand{\tm}{\ensuremath{\!\times\!}}

\newdir{|>}{{}*!/8pt/@{|}*!/3.5pt/:(1,-.2)@^{>}*!/3.5pt/:(1,+.2)@_{>}}
\newdir{ >}{{}*!/-10pt/@{>}}
\newdir{ |>}{!/10pt/{}*!/4.5pt/@{|}*:(1,-.2)@^{>}*:(1,+.2)@_{>}}
\newcommand\toiso{\overset{\sim\,}{\smash{\xrightarrow{\ \, }}\rule{0pt}{0.25ex}}}

\begin{document}

%
%
%
%
%
%
%
%
%

\title
{On pseudofunctors sending\\groups to 2-groups}

\author{Alan S.~Cigoli\footnote{Dipartimento di Matematica ``Giuseppe Peano'', Universit\`a degli Studi di Torino, via Carlo Alberto, 10, 10123 Torino, Italy},\ \  Sandra Mantovani\footnote{Dipartimento di Matematica ``Federigo Enriques'', Universit\`a degli Studi di Milano, via C.~Saldini, 50, 20133 Milano, Italy},\ Giuseppe Metere\footnote{Dipartimento di Matematica e Informatica, Universit\`a degli Studi di Palermo, via Archirafi, 34, 90123 Palermo, Italy}}


\maketitle

\begin{abstract}
For a category \cB\ with finite products, we first characterize pseudofunctors from \cB\ to \CAT\ whose associated opfibration is cartesian monoidal. Among those, we then characterize the ones which extend to pseudofunctors from internal groups to 2-groups. If \cB\ is additive, this is the case precisely when the associated opfibration has groupoidal fibres.

\bigskip
\noindent{\scriptsize {\bf Keywords}: pseudofunctor, internal groups, 2-groups, monoidal opfibration}

\noindent{\scriptsize {\bf MSC}: 18C40;18D30;18G45;18M05}
\end{abstract}


\section{Introduction} \label{sec:intro}

Let $\cB$ be a category with finite products, and let $F\colon \cB\to \Set$ be a finite product preserving functor.

It is well known that if
$(M,m\colon M\times M\to M,e\colon I_\cB\to M)$ is a monoid in $\cB$, then $F(M)$ is a monoid (in $\Set$) with multiplication $F(m)$ and identity $F(e)$. In fact, that assignment determines a functor $\Mon(\cB)\to \Mon$. In the same way, and with the same hypotheses, one gets  functors $\Gp(\cB)\to \Gp$ and $\Ab(\cB)\to \Ab$. All of them fit the commutative diagram below, where the vertical arrows are the usual forgetful functors.
\[
\xymatrix@R=2ex{
\Ab(\cB)\ar[d]\ar[rrr]^{\Ab(F)}&&&\Ab\ar[d]\\
\Gp(\cB)\ar[d]\ar[rrr]^{\Gp(F)}&&&\Gp\ar[d]\\
\Mon(\cB)\ar[d]\ar[rrr]^{\Mon(F)}&&&\Mon\ar[d]\\
\cB\ar[rrr]_-F&&&\Set
}
\]
Though elementary, the previous observation has relevant consequences, e.g.\ when the forgetful functor $\Ab(\cB)\to \cB$ is an isomorphism. If this is the case, then $F$ factors through the category $\Ab$ of abelian groups. This happens, for instance, when $\cB$ is the category of $C$-modules, for a fixed group $C$, and $F$ is a cohomology functor $H^n(C,-)$. For a given $C$-module $B$, thanks to the fact that $H^n(C,-)$ factors through $\Ab$, the set $H^n(C,B)$ gains the usual abelian group structure (see \cite{BJ}, where the Baer sum of abelian extensions in a protomodular category is introduced in this way).

A natural question arises, whether similar results hold true when we consider a 2-categorical setting rather then the 1-categorical one described above. In fact, the previous situation may be regarded as a monoidal functor $F\colon (\cB,\times,I_\cB)\to (\Set,\times,\{\ast\})$. As a generalization, we can consider a lax monoidal pseudofunctor $F\colon \cB\to \CAT$, where $\cB$ is regarded as a locally discrete 2-category (i.e.\ a 2-category whose 2-cells are identities), and both $\cB$ and $\CAT$ are endowed with their cartesian 2-monoidal structure. Then $F$ sends monoids in $\cB$ to \emph{pseudomonoids} in $\CAT$, i.e.\ to monoidal categories (\cite{DS1997}, see also \cite{MV2020}).
However, in general, a lax monoidal pseudofunctor fails to send (abelian) groups in $\cB$ to their 2-dimensional counterpart, namely \emph{(symmetric) 2-groups}: (symmetric) monoidal groupoids where every object is weakly invertible with respect to the tensor product. As an example, one can consider the pseudofunctor $\mathsf{Sub}(-)\colon \Ab \to \CAT$, that assigns to every abelian group $A$ the lattice of its subobjects. This pseudofunctor is indeed lax monoidal with respect to the cartesian structures and its fibres inherit a monoidal structure which is given by the join of subobjects, but they are clearly not 2-groups.

In the present paper, we first observe that pseudofunctors $F\colon \cB \to \CAT$ are always endowed with an oplax monoidal structure $(L,L^1)$ with respect to the cartesian monoidal structures (see Proposition \ref{prop:oplax_for_free}). Among those, in Theorem \ref{theorem:adj_make_prod} we characterize the ones whose corresponding opfibration (obtained via the Grothendieck construction) is a cartesian monoidal opfibration. This happens precisely when $L$ and $L^1$ have right adjoints $R$ and $R^1$ respectively.

At this point, our goal is to understand under which conditions such pseudofunctors lift to pseudofunctors from internal groups to 2-groups. We call the latter \emph{groupal} pseudofunctors (see Definition \ref{def:gps}) and we fully answer this question in Theorem \ref{thm:Gp(B)}: $F$ is groupal if and only if the unit $\eta^1\colon 1 \Rightarrow R^1L^1$ is an isomorphism as well as the units $\eta^{A,A}\colon 1 \Rightarrow R^{A,A}L^{A,A}$, for each group object $A$ in \cB. This can be rephrased in terms of the opfibration $P$: thanks to Lemma \ref{lemma:unit_terminal} and Lemma \ref{lemma:unit_product}, the previous conditions hold as soon as, for each object $X$ with $P(X)$ a group object in \cB, the terminal map $\tau_X\colon X \to I_\cX$ and the diagonal map $\Delta_X\colon X \to X\times X$ are cocartesian. In particular, in Corollary \ref{cor:groupoidal_fibres} we prove that a cartesian monoidal opfibration with groupoidal fibres always satisfies these properties, and then the corresponding pseudofunctor is groupal.

As a special case of interest, we apply the previous results to a cartesian monoidal opfibration $P$ with additive base. The corresponding pseudofunctor factors through $\SCG$ if and only if diagonal and terminal maps are $P$-cocartesian (Theorem \ref{thm:additive}). This happens precisely when the fibres of $P$ are groupoids (Theorem \ref{thm:final}).

In Section \ref{sec:H^2}, we consider the pseudofunctor $\OPEXT(C,-)\colon \Mod(C) \to \CAT$, associating with any $C$-module $(B,\xi)$ (where $C$ is a group) the groupoid $\OPEXT(C,B,\xi)$  of abelian extensions of $C$, with kernel $B$ and induced action $\xi$. We can apply Theorem \ref{thm:final}, so that each such groupoid gets a symmetric 2-group structure, whose connected components  give rise to a group $\pi_0(\OPEXT(C,B,\xi))$ isomorphic to   $H^2(C,B,\xi)$ (and  $\pi_1(\OPEXT(C,B,\xi))$  isomorphic to  $Z^1(C,B,\xi)).$ In this way, we recover the result achieved by Vitale in \cite{Vitale03} in terms of cocycles. Another viewpoint can be found in Section 6.2 of \cite{CM16}.

As a further  application, in Section \ref{sec:torsors} we consider the category $\mathsf{LMSet}$ of monoid left actions, taking into account the fact that the forgetful functor $V\colon\mathsf{LMSet}\to\Mon$ turns out to be a cartesian monoidal opfibration. By Proposition \ref{prop:62}, its associated pseudofunctor  lifts to a lax symmetric monoidal  pseudofunctor $\ACT\colon \CMon\to \SMC$, where the tensor product in any fibre $\ACT(M)$ can be obtained by means of the so-called contracted product \cite{Breen} (see Proposition \ref{prop:contracted}).

$\ACT(M)$ fails to be a groupoid, even when $M$ is an (abelian) group. Thanks to Proposition \ref{prop:64}, it turns out that the largest possible restriction of $V$ over $\Ab$ with groupoidal fibres is given by $\overline V \colon \mathsf{ATors} \to \Ab$, where $\mathsf{ATors}$ is the full subcategory of $\mathsf{LMSet}$ given by abelian torsors. Then, thanks to Theorem \ref{thm:final}, the corresponding pseudofunctor $\mathsf{TORS}$ factors through symmetric 2-groups. Notice that, for an abelian group $B$, the structure we obtain on $\mathsf{TORS}(B)$ is a special case of the one studied for bitorsors by Breen in \cite{Breen}.

\section{Preliminaries} \label{sec:preliminaries}

We use unadorned capital letters for categories ($\cA, \cB$ etc.), and thick ones for 2-categories ($\bA, \bB$ etc.). $\mathbf{CAT}$ is the 3-category of 2-categories, pseudo\-functors, pseudonatural transformations and modifications. We denote by $\CAT$ the 2-category of locally small categories. Actually, we will avoid dealing with size issues, which can be resolved with suitable Grothendieck universes. In order to fix this point, we declare that the category of small categories $\Cat$ is an object of $\CAT$, which is in turn an object of $\mathbf{CAT}$. Composition of arrows (1-cells) $f, g$ is denoted $g\circ f$, or just $gf$, and the same notation is adopted for the whiskering of a 1-cell with a 2-cell. Horizontal composition of 2-cells is always denoted $\beta\circ\alpha$, since the juxtaposition is reserved to vertical composition of 2-cells.

Binary products  in a category $\cB$ are denoted as usual, with projections $\pi_1,\pi_2$, and the terminal map from an object $A$ is $\tau^\cB_A\colon A\to I_\cB$, with superscripts and subscripts omitted when clear from the context.
We adopt similar conventions for finite products in $\CAT$, where the terminal category with unique object $\star$ is denoted  by $\cI$. In fact, finite products in $\CAT$ are also 2-limits, in that they also satisfy a universal property with respect to 2-cells. We will refer to it as \emph{the 2-dimensional universal property of products}, to make a distinction from the usual (1-dimensional) universal property of products. In the case of binary products, for two categories $\cA$ and $\cB$, their product
$\xymatrix{\cA&\ar[l]_-{\pi_1}\cA\times\cB\ar[r]^-{\pi_2}&\cB}$ satisfies:
for any pair of natural transformations $\alpha\colon a\Rightarrow a'\colon \cC\to\cA$ and $\beta\colon b\Rightarrow b'\colon \cC\to\cB$, there exists a unique natural transformation $\gamma\colon c\Rightarrow c'\colon\cC\to\cA\times\cB$ such that $\pi_1\circ \gamma =\alpha$ and $\pi_2\circ \gamma =\beta$.

Let $\cB$ be a category. A pseudofunctor $F=(F,\phi,\phi^1)\colon \cB\to \CAT$ is a weak 2-functor, where $\cB$ is regarded as a locally discrete 2-category, and $\phi_{f,g}\colon F(g)\circ F(f)\toiso F(g\circ f)$, $\phi^1_B\colon 1_{F(B)}\toiso F(1_B)$ are natural isomorphisms, for $f,g$ composable arrows and $B$ object of $\cB$, satisfying coherence conditions.

\begin{Example}
The following is classical. Let $P\colon\cX\to\cB$ be a (Grothendieck) cloven opfibration. Then it is possible to associate with $P$ the pseudofunctor $F\colon \cB\to\CAT$ defined as follows. For any object $A$ of $\cB$, the image $F(A)$ is the fibre of $P$ over $A$. For an arrow $f\colon A\to B$ of $\cB$, $F(f)=f_*\colon F(A)\to F(B)$ is the change of base functor (determined by the cleavage), where for any $X\in F(A)$, $\hat{f}\colon X\to f_*(X)$ is a cocartesian lifting of $f$ at the object $X$.
The natural isomorphisms $(\phi,\phi^1)$ are consequences of the universal properties of cocartesian liftings.

\emph{Notation}: in the case of composable arrows $f,g$, for the codomain of a composition of liftings, we write $g_*f_*(X)$ rather than $g_*(f_*(X))$.
\end{Example}

Monoidal bicategories are the bicategorical version of monoidal categories, and they can be described as one-object tricategories (\cite{GPS1995}, see also \cite{DS1997}).  They can be defined according to different levels of strictness, concerning the underlying bicategory and the tensor product. For instance, we say that $(\bB,\otimes,I)$ is a  \emph{monoidal 2-category} if the underlying bicategory $\bB$ is in fact a 2-category.
The interested reader can also consult \cite{McC2000}, where monoidal 2-categories and their variations are described in details.
However, we do not need to recall the  definitions in full generality, since in this paper we will stick to the two evident examples  described below.

\begin{Example}
Let $(\cB, \otimes, I_\cB)$ be a symmetric monoidal category. Then, when $\cB$ is regarded as a locally discrete 2-category, it can be likewise considered as a symmetric monoidal 2-category.
\end{Example}

\begin{Example}
The 2-category $\CAT$ can be endowed with a cartesian monoidal structure given by the cartesian product and the terminal category, which makes it a symmetric monoidal 2-category.
\end{Example}

Morphisms between monoidal bicategories are homomorphisms of bicategories that preserve the structures up to coherent pseudonatural transformations. Once again, we describe explicitly only the case of interest in our context.
We start by fixing the notation for an oplax monoidal pseudofunctor $\cB\to \CAT$, where the monoidal structure on $\cB$ is the cartesian one.

Let $\cB$ be a category with finite products. The associator is induced by the universal property of products, and it is written $\alpha_{A,B,C}\colon (A\times B)\times C\to A\times(B\times C)$, left and right unitors are the projections $\pi_1\colon A\times I_\cB\to A$ and $\pi_2\colon I_\cB\times A\to A$. The same conventions are adopted for the constraints of the cartesian monoidal structure of $\CAT$.

An oplax monoidal pseudofunctor
\[
F=(F,L,L^1,\omega,\zeta,\xi)\colon (\cB,\times,I_\cB)\longrightarrow (\CAT,\times,\cI)
\]
is given by a pseudofunctor $F\colon\cB\to \CAT$, together with pseudonatural transformations
\begin{equation} \label{diag:oplax_monoidal}
\begin{aligned}
\xymatrix{
\cB\times\cB\ar[d]_{F\times F}\ar[r]^-{\times}
&\cB\ar[d]^{F}\ar@{}[dl]|(.35){}="1"|(.65){}="2"
\ar@{=>}"1";"2"^{L}
\\
\CAT\times\CAT\ar[r]_-{\times}
&\CAT}
\qquad
\xymatrix@C=10ex{
\bI\ar[r]^{I_{\cB}}\ar[dr]_{\cI}
&\cB\ar[d]^{F}\ar@{}[dl]|(.19){}="1"|(.42){}="2"
\ar@{=>}"1";"2"^(.3){L^1}
\\
&\CAT
}
\end{aligned}
\end{equation}
(where $\bI$ is the terminal 2-category) with components the functors
\begin{equation}\label{diag:oplax_monoidal_components}
L^{A,B}\colon F(A\times B)\to F(A)\times F(B)\qquad L^1\colon F(I_{\cB})\to \cI
\end{equation}
and invertible modifications with components
\begin{equation}\label{diag:oplax_monoidal_mod_1}
\begin{aligned}
\xymatrix@C=16ex{
F((A\times B)\times C)\ar[d]_{L^{A\times B, C}}\ar[r]^{F(\alpha_{A,B,C})}
&F(A\times(B\times C))\ar[d]^{L^{A,B\times C}}\ar@{}[ddl]|(.4){}="1"|(.6){}="2"
\ar@{=>}"1";"2"^(.3){\ \ \omega_{A,B,C}}_{\cong}
\\
F(A\times B)\times F(C)\ar[d]_{L^{A,B}\times 1}
&F(A)\times F(B\times C)\ar[d]^{1\times L^{B,C}}
\\
(F(A)\times F(B))\times F(C)\ar[r]_{\alpha_{F(A),F(B),F(C)}}
&F(A)\times (F(B)\times F(C))
}
\end{aligned}
\end{equation}
\begin{equation}\label{diag:oplax_monoidal_mod_2}
\begin{aligned}
\xymatrix@C=10ex{
F(A\times I_\cB)\ar[d]_{L^{A,I_\cB}}\ar[rr]^{F(\pi_1)}
&&F(A)\ar[d]^1
\ar@{}[dll]|(.4){}="1"|(.6){}="2"
\ar@{=>}"1";"2"^(.3){\ \ \zeta_{A}}_{\cong}
\\
F(A)\times F(I_\cB)\ar[r]_{1\times L^1}
&F(A)\times \cI\ar[r]_{\pi_1}
&F(A)
}
\end{aligned}
\end{equation}
\begin{equation}\label{diag:oplax_monoidal_mod_3}
\begin{aligned}
\xymatrix@C=10ex{
F(I_\cB\times A)\ar[d]_{L^{I_\cB,A}}\ar[rr]^{F(\pi_2)}
&&F(A)\ar[d]^1
\ar@{}[dll]|(.4){}="1"|(.6){}="2"
\ar@{=>}"1";"2"^(.3){\ \ \xi_{A}}_{\cong}
\\
F(I_\cB)\times F(A)\ar[r]_{ L^1\times 1}
&\cI\times F(A)\ar[r]_{\pi_2}
&F(A)
}
\end{aligned}
\end{equation}
satisfying suitable coherence conditions.

If we write $\sigma_{A,B}\colon A\times B\to B\times A$ for the braiding induced by a cartesian monoidal structure, the oplax monoidal pseudofunctor $F$ is oplax  \emph{symmetric} monoidal if it is endowed with an invertible modification $\theta$ with components
\begin{equation}\label{diag:oplax_monoidal_mod_4}
\begin{aligned}
\xymatrix@C=10ex{
F(A\times B)\ar[r]^{F(\sigma_{A,B})}\ar[d]_{L^{A,B}}
&F(B\times A)\ar[d]^{L^{B,A}}
\ar@{}[dl]|(.4){}="1"|(.6){}="2"
\ar@{=>}"1";"2"^(.3){\ \ \theta_{A,B}}_{\cong}
\\
F(A)\times F(B)\ar[r]_{\sigma_{F(A),F(B)}}
&F(B)\times F(A)
}
\end{aligned}
\end{equation}
satisfying suitable coherence conditions.

Dually, a lax monoidal pseudofunctor
\[
F=(F,R,R^1,\omega',\zeta',\xi')\colon (\cB,\times,I_\cB)\longrightarrow (\CAT,\times,\cI)
\]
is given by a pseudofunctor $F\colon\cB\to \CAT$, together with pseudonatural transformations
\begin{equation}\label{diag:lax_monoidal}
\begin{aligned}
\xymatrix{
\cB\times\cB\ar[d]_{F\times F}\ar[r]^-{\times}
&\cB\ar[d]^{F}\ar@{}[dl]|(.35){}="1"|(.65){}="2"
\ar@{=>}"2";"1"_{R}
\\
\CAT\times\CAT\ar[r]_-{\times}
&\CAT}
\qquad
\xymatrix@C=10ex{
\bI\ar[r]^{I_{\cB}}\ar[dr]_{\cI}
&\cB\ar[d]^{F}\ar@{}[dl]|(.19){}="1"|(.42){}="2"
\ar@{=>}"2";"1"_(.6){R^1}
\\
&\CAT
}
\end{aligned}
\end{equation}
with components the functors
\begin{equation}\label{diag:lax_monoidal_components}
R^{A,B}\colon  F(A)\times F(B)\to F(A\times B)\qquad R^1\colon  \cI\to F(I_{\cB})
\end{equation}
and invertible modifications $\omega', \zeta', \xi'$ whose components can be easily guessed from those of the oplax case. Similarly, one defines a lax symmetric monoidal pseudofunctor.

Pseudomonoids in a monoidal bicategory are introduced in Section~3 of \cite{DS1997} (see also \cite{McC2000} for a more contemporary and slightly more general account). However, we will need some details just for the case of pseudomonoids in a monoidal 2-category. A \emph{pseudomonoid} in a monoidal $2$-category $(\bB,\otimes,I)$ is an internal monoid object in $\bB$, such that the monoid axioms hold only up to invertible 2-cells. This means that it is an object $M$ of $\bB$, together with 1-cells $m\colon M\otimes M\to M$ and $e\colon I\to M$ and invertible 2-cells
\begin{equation} \label{diag:pseudomonoid1}
\begin{aligned}
\xymatrix{
(M\otimes M)\otimes M\cong M\otimes (M\otimes M)\ar[d]_{m\otimes 1}\ar[r]^-{1\otimes m}
&M\otimes M\ar[d]^m
\ar@{}[dl]|(.4){}="1"|(.6){}="2"
\ar@{=>}"1";"2"^(.3){\alpha}_{\cong}
\\
M\otimes M\ar[r]_{m}
&M
}
\end{aligned}
\end{equation}
\begin{equation} \label{diag:pseudomonoid2}
\begin{aligned}
\xymatrix{
M\cong I\otimes M\ar[r]^-{1\otimes e}\ar[dr]_1
&M\otimes M\ar[d]^{m}
\ar@{}[dl]|(.2){}="1"|(.4){}="2"
\ar@{=>}"1";"2"^(.3){\lambda}_{\cong}
\ar@{}[dr]|(.2){}="1"|(.4){}="2"
\ar@{=>}"1";"2"_(.3){\rho}^{\cong}
&M\otimes I\cong M\ar[l]_-{e\otimes 1}\ar[dl]^1
\\&M&
}
\end{aligned}
\end{equation}
satisfying coherence conditions. A (op)lax morphism between pseudomonoids $M$ and $N$ is a 1-cell $f\colon M\to N$ together with 2-cells that express the (op)lax preservation of the internal multiplication and unit.

Most of the pseudomonoids considered in this paper are symmetric. In order to define symmetric pseudomonoids, the ambient 2-category cannot be only monoidal, but must be \emph{sylleptic monoidal}. We do not enter into detail, since in the cartesian monoidal cases we consider, the sylleptic structure is the canonical symmetric one, as determined by universal property of products. However, the only case we shall need in order to make explicit calculations are described in the following two examples.

\begin{Example}
Let $\cB$ be a category with finite products, considered as a monoidal 2-category with its cartesian symmetric monoidal structure. A (symmetric) pseudomonoid in $(\cB, \times, I_\cB)$  is nothing but a (commutative) mo-noid object in $\cB$, a (op)lax morphism is just a morphism of internal monoids.
\end{Example}

\begin{Example}
A (symmetric) pseudomonoid in the cartesian monoidal 2-category $(\CAT,\times,\cI)$ is precisely a (symmetric) monoidal category. In this case, a (op)lax morphism is a (op)lax monoidal functor.
\end{Example}

In \cite{MV2020}, the authors characterize monoidal fibrations (defined in \cite{Shu2008}) as pseudomonoids internal to the 2-category of fibrations (in fact they \emph{define} monoidal fibrations in this way, and we follow such convention), and then they show how the well-known biequivalence between fibrations and indexed categories extends to internal pseudomonoids in such monoidal 2-categories. Here we recall their definition/characterization, but spelled out in the opfibration case, which is more suitable for our purposes.
\begin{Proposition}\label{prop:Prop_3.2_MV}[see Proposition~3.2 of \cite{MV2020}]
The following statements are equivalent:
\begin{itemize}
\item[i.] The functor $P\colon \cX\to\cB$ is a monoidal opfibration, i.e.\ it is a pseudomonoid in the 2-category $\opFib$ of opfibrations.
\item[ii.] The categories $\cX$ and $\cB$ are monoidal, the opfibration $P$ is a strict monoidal functor, and the tensor product of cocartesian maps is cocartesian.
\end{itemize}
\end{Proposition}
Furthermore, all of this can be translated in terms of pseudofunctors (see  Proposition~3.7 of \cite{MV2020}): the pseudofunctor $F$ canonically associated with an opfibration $P$ satisfying the equivalent conditions above is endowed  with a lax monoidal structure
\[
F=(F,R,R^1,\dots)\colon (\cB,\otimes,I_\cB)\longrightarrow (\CAT,\times,\cI)\,.
\]
We recall from \cite{MV2020} that this correspondence extends to an equivalence of 2-categories.

Starting with a lax monoidal pseudofunctor $F$ as above, it is easy to see that it sends (commutative) monoids in $\cB$ to (symmetric) monoidal categories. In order to fix the notation, we recall such a procedure. Let $(M,m,e)$ be a monoid in $\cB$. Then,  $F(M)$ becomes a monoidal category with tensor product and unit object
\begin{equation}\label{diag:F(M)tensor}
\xymatrix{
\otimes\colon F(M)\times F(M)\ar[r]^-{R^{M,M}}&F(M\otimes M)\ar[r]^-{F(m)}&F(M)
}
\end{equation}
\begin{equation}\label{diag:F(M)unit}
\xymatrix{
\cI\ar[r]^-{R^1}&F(I_\cB)\ar[r]^-{F(e)}&F(M)
}
\end{equation}
Also, for a morphism of (commutative) monoids  $f\colon (M,m,e)\to (M',m',e')$ in $\cB$, $F(f)$ is a strong (symmetric) monoidal functor $F(M) \to F(M')$.

\section{When oplax monoidal is for free}

Let $\cB$ be a category with finite products, $P\colon \cX\to \cB$ an opfibration and $F\colon \cB\to \CAT$ its associated pseudofunctor.

\begin{Lemma} \label{lemma:oplax_unit}
The trivial functor $L^1\colon F(I_\cB)\to \cI$ defines a pseudonatural transformation $L^1$  as in the diagram on the right of \eqref{diag:oplax_monoidal}.
\end{Lemma}

\begin{proof}
Trivial.
\end{proof}

\begin{Lemma} \label{lemma:oplax_product_1}
Let $A,B$ be objects of $\cB$. The assignment
\[
\begin{aligned}
\xymatrix{X\ar[d]_f\\Y}
\end{aligned}
\quad\mapsto\quad
\left(
\begin{aligned}
\xymatrix{\pi_{1*}(X)\ar[d]_{\pi_{1*}(f)}\\\pi_{1*}(Y)}
\end{aligned}
,\begin{aligned}
\xymatrix{\pi_{2*}(X)\ar[d]^{\pi_{2*}(f)}\\\pi_{2*}(Y)}
\end{aligned}
\right)
\]
defines a functor $L^{A,B}\colon F(A\times B)\to F(A)\times F(B)$.
\end{Lemma}

\begin{proof}
The functor $L^{A,B}$ is in fact defined by the (1-dimensional) universal property of products in $\CAT$, as the following diagram illustrates:
\[
\xymatrix{
&F(A\times B)\ar[dl]_{F(\pi_1)}\ar@{-->}[d]^{L^{A,B}}\ar[dr]^{F(\pi_1)}
\\
F(A)&F(A)\times F(B)\ar[l]^-{\pi_1}\ar[r]_-{\pi_2}
&F(B)}
\]
In the language of opfibrations, $F(\pi_i)=\pi_{i*}$ and  $L^{A,B}=\langle \pi_{1*},\pi_{2*} \rangle$.
\end{proof}

\begin{Lemma} \label{lemma:oplax_product_2}
The functor $L^{A,B}$ defines a pseudonatural transformation $L$ as in the diagram on the left of \eqref{diag:oplax_monoidal}.
\end{Lemma}

\begin{proof}
Pseudonaturality is given by the 2-dimensional property of products in $\CAT$. More precisely, for arrows $a\colon A\to A'$ and $b\colon B\to B'$, by composing the (not necessarily commutative) square
\begin{equation}\label{diag:L^(a,b)}
\begin{aligned}
\xymatrix@C=16ex{
F(A\times B)\ar[d]_{L^{A,B}}\ar[r]^{F(a\times b)}
&F(A'\times B')\ar[d]^{L^{A'\!,B'}}
\ar@{}[dl]|(.4){}="1"|(.6){}="2"
\ar@{==>}"1";"2"^(.3){\ \ L^{(a,b)}}_{\cong}
\\
F(A)\times F(B)\ar[r]_{F(a)\times F(b)}
&F(A')\times F(B')}
\end{aligned}
\end{equation}
with product projections, one obtains the squares
\[
\xymatrix@C=10ex{
F(A\times B)\ar[d]_{F(\pi_1)}\ar[r]^-{F(a\times b)}
&F(A'\times B')\ar[d]^{F(\pi_1)}
\ar@{}[dl]|(.4){}="1"|(.6){}="2"
\ar@{=>}"1";"2"^(.3){\ \ \iota_1}_{\cong}
\\
F(A)\ar[r]_-{F(a)}
&F(A')}
\xymatrix@C=10ex{
F(A\times B)\ar[d]_{F(\pi_2)}\ar[r]^-{F(a\times b)}
&F(A'\times B')\ar[d]^{F(\pi_2)}
\ar@{}[dl]|(.4){}="1"|(.6){}="2"
\ar@{=>}"1";"2"^(.3){\ \ \iota_2}_{\cong}
\\
F(B)\ar[r]_-{F(b)}
&F(B')}
\]
Due to the pseudofunctoriality of $F$, these are filled with natural isomorphisms $\iota_1$ and $\iota_2$. The universal property produces the natural isomorphism $L^{(a,b)}$ which fills diagram \eqref{diag:L^(a,b)}, such that $\pi_i\circ L^{(a,b)}= \iota_i$.
\end{proof}

It is useful to spell out the definition of $L^{(a,b)}$ on the components. For an object $Z$ of $F(A\times B)$, if we write $L^{(a,b)}_Z=(L^{(a,b)}_{Z,1},L^{(a,b)}_{Z,2})$, then the isomorphism $L^{(a,b)}_{Z,1}\colon a_*\pi_{1*}(Z)\to \pi_{1*}(a\times b)_*(Z)$ is the vertical comparison between two different cocartesian liftings at $Z$ of  $a\circ \pi_1=\pi_1\circ (a\times b)\colon A\times B\to A'$.
Similarly for the isomorphism $L^{(a,b)}_{Z,2}\colon b_*\pi_{2*}(Z)\to \pi_{2*}(a\times b)_*(Z)$.

\begin{Proposition} \label{prop:oplax_for_free}
Let $\cB$ be a category with finite products, and $F\colon \cB\to \CAT$ be a  pseudofunctor. Then $F$ is canonically endowed with an oplax symmetric monoidal structure
\[
(F,L,L^1,\omega,\zeta,\xi,\theta)\colon (\cB,\times,I_{\cB})\longrightarrow (\CAT,\times,\cI)
\]
\end{Proposition}

\begin{proof}
Oplax preservation of unit object is given by Lemma~\ref{lemma:oplax_unit}, while
oplax preservation of products by Lemma~\ref{lemma:oplax_product_1} and by Lemma~\ref{lemma:oplax_product_2}.
It is tedious, but essentially straightforward, to obtain the invertible modifications $\omega,\zeta,\xi$ and $\theta$ of diagrams \eqref{diag:oplax_monoidal_mod_1}, \eqref{diag:oplax_monoidal_mod_2}, \eqref{diag:oplax_monoidal_mod_3} and \eqref{diag:oplax_monoidal_mod_4}: a careful application of the 2-dimensional property of products in $\CAT$ will do the job. In the same way, one proves that coherence conditions are satisfied.
\end{proof}

\section{Adjoints make products}

\begin{Lemma} \label{lemma:R1}
Let $\cB$ be a category with terminal object $I_\cB$ and $P\colon \cX\to\cB$ be an opfibration, and $F\colon \cB\to\CAT$ its associated pseudofunctor. Then the following statements are equivalent.
\begin{itemize}
\item[i.] The category $\cX$ has a terminal object $I_\cX$, and $P$ strictly preserves it.
\item[ii.] The functor $L^1$ defined in Lemma~\ref{lemma:oplax_unit} has a right adjoint $R^1$.
\end{itemize}
\end{Lemma}

\begin{proof}
$(i.\Rightarrow ii.)$ Since $P(I_\cX)=I_\cB$ by hypothesis, we can define $R^1$ by $R^1(\star)=I_\cX$. $R^1$ is trivially right adjoint to $L^1$: the unique component of the counit is $\epsilon^1_\star=1_\star\colon L^1R^1(\star)=\star\to\star$. And, for $X\in F(I_\cB)$, $\eta^1_X\colon X\to R^1L^1(X)=I_\cX$ is the terminal map $\tau_X$.

$(ii.\Rightarrow i.)$ Since $R^1$ is a right adjoint, $R^1(\star)$ is terminal in $F(I_\cB)$.  We claim it is also  terminal in $\cX$.
Indeed, let $X$ be any object of $\cX$. We can lift the terminal map $\tau_{P(X)}\colon P(X)\to I_\cB$ at $X$, and then compose this lifting $\hat\tau_{P(X)}$ with the terminal map $t\colon \tau_{P(X)*}(X)\to I_\cX=R^1(\star)$ in $F(I_\cB)$. By the property of cocartesian liftings, this is the unique arrow $X\to I_\cX$.
\end{proof}

\begin{Lemma} \label{lemma:R}
Let $\cB$ be a category with binary products and $P\colon \cX\to\cB$ be  an opfibration, and $F\colon \cB\to\CAT$ its associated pseudofunctor.  For $A$ and $B$ objects of $\cB$, the following statements are equivalent.
\begin{itemize}
\item[i.]  For $X$ and $Y$ in $\cX$ such that $P(X)=A$ and $P(Y)=B$, there exists a product $X\times Y$ in $\cX$ such that $P(X\times Y)=A\times B$.
\item[ii.] The functor $L^{A,B}$ defined in Lemma~\ref{lemma:oplax_product_1} has a right adjoint $R^{A,B}$.
\end{itemize}
If this is the case, we have $X\times Y=R^{A,B}(X,Y)$.
\end{Lemma}
\begin{proof} $(i.\Rightarrow ii.)$ Let $X,Y$ be objects of $\cX$, and $A=P(X)$ and $B=P(Y)$. We can define $R^{A,B}(X,Y)=X\times Y$, and this easily extends to a functor
\[
R^{A,B}\colon F(A)\times F(B) \longrightarrow F(A\times B)
\]
by the universal property of products in $\cX$. In order to prove that $R^{A,B}$ is right adjoint to $L^{A,B}$, let us define the counit $\epsilon^{A,B}\colon L^{A,B}R^{A,B}\Rightarrow 1_{F(A)\times F(B)}$  and show that it satisfies the expected universal property.

The counit is obtained with the help of the universal property of cocartesian liftings of product projections in $\cB$. For $X, Y$ in $\cX$, the component
\[
\epsilon^{A,B}_{(X,Y)}=(\epsilon_1,\epsilon_2)\colon L^{A,B}R^{A,B}(X,Y)=(\pi_{1*}(X\times Y),\pi_{2*}(X\times Y))\longrightarrow (X,Y)
\]
of the counit is defined as follows: $\epsilon_i$ ($i=1,2$) are the unique arrows such that $\epsilon_i\circ \hat\pi_i=\pi_i$ and $P(\epsilon_i)$ are identities -- the reader may find it helpful to look at the lower part of diagram \eqref{diag:counit}. Moreover, $(\epsilon_1,\epsilon_2)$ satisfies the universal property of the counit.
Actually, given the arrow
\[
(g_1,g_2)\colon L^{A,B}(Z)=(\pi_{1*}(Z),\pi_{2*}(Z))\longrightarrow(X,Y)
\]
in $F(A)\times F(B)$, there exists a unique $h$ such that $\pi_i\circ h=g_i\circ \hat\pi_i$ (see  diagram \eqref{diag:counit} below, where the dotted lines help to identify the fibres):
\begin{equation}\label{diag:counit}
\begin{aligned}
\xymatrix@C=10ex{
\pi_{1*}(Z)\ar@{-->}[d]^{\pi_{1*}(h)}\ar@/_8ex/[dd]_{g_1}
&\ar[l]_-{\hat\pi_1}Z\ar[r]^-{\hat\pi_2}\ar@{-->}[d]^{h}
&\pi_{2*}(Z)\ar@{-->}[d]_{\pi_{2*}(h)}\ar@/^8ex/[dd]^{g_2}
\\
\pi_{1*}(X\times Y)\ar[d]_{\epsilon_1}
&\ar[l]_-{\hat\pi_1}X\times Y\ar@{.}[dd]\ar[r]^-{\hat\pi_2}\ar[dl]^{\pi_1}\ar[dr]_{\pi_2}
&\pi_{2*}(X\times Y)\ar[d]^{\epsilon_2}
\\
X\ar@{.}[d]&&Y\ar@{.}[d]
\\
A&\ar[l]^{\pi_1}A\times B\ar[r]_{\pi_2}&B
}
\end{aligned}
\end{equation}
Then, we can apply $L^{A,B}$ to $h$, and obtain the pair $(\pi_{1*}(h),\pi_{2*}(h))$. Finally,  $g_i=\epsilon_i\circ \pi_{i*}(h)$, since they are $P$-vertical, and precomposing with the cocartesian arrow $\hat\pi_i$ gives the same result.

Notice that the corresponding unit component in $Z$ is given by
\[
\eta^{A,B}_{Z}=\langle \hat\pi_1,\hat\pi_2  \rangle\colon Z\longrightarrow \pi_{1*}(Z)\times\pi_{2*}(Z)\,,
\]
where $\hat\pi_i$ is the cocartesian lifting at $Z$ of the product projection $\pi_i$.

$(ii.\Rightarrow i.)$ Let  $X,Y$ be objects of $\cX$, and let $A=P(X)$ and $B=P(Y)$. We want to prove that  $R^{A,B}(X,Y)$  gives rise to a product $X\times Y$ in $\cX$, with   projections  $\pi_i=\epsilon_i\circ \hat\pi_i$, where $\epsilon^{A,B}_{(X,Y)}=(\epsilon_1,\epsilon_2)$ is the counit of the adjunction $L^{A,B}\dashv  R^{A,B}$, and $\hat\pi_i$ is the cocartesian lifting of the product projection $\pi_i$. In order to prove that we actually define a product, let us consider two arrows $g_1\colon Z\to X$ and $g_2\colon Z\to Y$ in $\cX$ and their images  $P(g_1)=f_1\colon C\to A$ and $P(g_2)=f_2\colon C\to B$ in $\cB$. By taking the cocartesian lifting of $\langle f_1,f_2\rangle\colon C\to A\times B$ at $Z$, we obtain that $g_i= g'_i\circ \widehat{\langle f_1,f_2\rangle}$, where $g'_i$ is the unique factor such that $P(g'_i)=\pi_i$:
\[
\xymatrix{
\pi_{1 *}(X\times Y)\ar@/_4ex/[dd]_{\epsilon_1}
&&X\times Y\ar[ll]_-{\hat\pi_1}
\ar[rr]^-{\hat\pi_2}
&&\pi_{2 *}(X\times Y)\ar@/^4ex/[dd]^{\epsilon_2}
\\
W_1\ar[d]^{k_1}\ar@{-->}[u]_{\pi_{1*}(k)}
&&\langle f_1,f_2\rangle_*(Z)\ar@{-->}[ldl]_{g'_1}\ar@{-->}[rrd]^{g'_2}\ar@{.}[ddd]\ar@{-->}[u]_{k}
\ar[ll]_{\hat\pi_1}\ar[rr]^{\hat\pi_2}
&&W_2\ar[d]_{k_2}\ar@{-->}[u]^{\pi_{2*}(k)}
\\
X\ar@{.}[dd]
&&
&&Y\ar@{.}[dd]
\\
&&&Z\ar[ulll]^{g_1} \ar[uul]_(.35){\widehat{\langle f_1,f_2\rangle}} \ar@{.}[dd] \ar[ur]_{g_2}
\\
A
&&A\times B\ar[ll]_{\pi_1} \ar[rr]^(.65){\pi_2}
&&B
\\
&&&C\ar[ulll]^{f_1} \ar[ul]_{\langle f_1,f_2\rangle}  \ar[ur]_{f_2}
}
\]
Then there exists a unique $k_i$ such that $P(k_i)=1$ and $k_i\circ \hat\pi_i=g'_i$, where $\hat\pi_i$ is the cocartesian lifting at $\langle f_1,f_2\rangle_*(Z)$  of $\pi_i$.
So we get a morphism $(k_1,k_2)\colon (W_1,W_2)=L^{A,B}(\langle f_1,f_2\rangle_*(Z))\to(X,Y)$. Then, by the universal property of counits, there exists a unique $k\colon\langle f_1,f_2\rangle_*(Z)\to X\times Y$ such that $\epsilon_i\circ \pi_{i*}(k)=k_i$.
By precomposing, we eventually find the unique arrow $k \circ \widehat{\langle f_1,f_2\rangle}\colon Z\to X\times Y$ such that
\[
\epsilon_i\circ \hat \pi_i \circ k \circ \langle f_1,f_2\rangle
=\epsilon_i\circ \pi_{i*}(k) \circ  \hat \pi_i \circ \langle f_1,f_2\rangle
=k_i\circ  \hat \pi_i \circ \langle f_1,f_2\rangle
=g'_i\circ \langle f_1,f_2\rangle=g_i.
\]
\end{proof}

\begin{Corollary} \label{corollary:R}
Let $\cB$ be a category with binary products and $P\colon \cX\to\cB$ be  an opfibration, and $F\colon \cB\to\CAT$ its associated pseudofunctor. Then the following statements are equivalent.
\begin{itemize}
\item[i.]  The category $\cX$ has binary products, and $P$ strictly preserves them.
\item[ii.] For every pair of objects $A$ and $B$ in $\cB$, the functor $L^{A,B}$ defined in Lemma~\ref{lemma:oplax_product_1} has a right adjoint $R^{A,B}$.
\end{itemize}
If this is the case, we have $X\times Y=R^{P(X),P(Y)}(X,Y)$.
\end{Corollary}

\begin{proof}
Just apply Lemma~\ref{lemma:R} for all $A,B$ in $\cB$.
\end{proof}

\begin{Remark}
Put together, Lemma~\ref{lemma:R1} and Corollary~\ref{corollary:R} say that requiring \ref{lemma:R1}.\emph{i}.\ plus \ref{corollary:R}.\emph{i}.\ to hold is the same as requiring $\cX$ to have finite products, and $P$ to be a strict monoidal functor with respect to the cartesian structures.
\end{Remark}

For a pair $(a,b)$ of arrows in $\cB$, let us consider the natural isomorphism $L^{(a,b)}$  described in diagram~\eqref{diag:L^(a,b)}. If $L^{A,B}$ has a right adjoint $R^{A,B}$ and $L^{A'\!,B'}$ has right adjoint $R^{A'\!,B'}$, then the pasting
\begin{equation}\label{diag:mate}
\begin{aligned}
\xymatrix@C=3.2ex{
F(A)\tm F( B)\ar[r]^-{R^{A,B}}\ar[ddr]_1
&F(A\tm B)\ar[dd]^{L^{A,B}}\ar[rr]^{F(a\times b)}
\ar@{}[ddl]|(.25){}="1"|(.4){}="2"
\ar@{=>}"1";"2"_(.3){\epsilon^{A,B}}
&&F(A'\tm B')\ar[dd]_{L^{A'\!,B'}\!\!}
\ar@{}[ddll]|(.3){}="1"|(.7){}="2"
\ar@{=>}"1";"2"^(.5){ L^{(a,b)}}_-{\cong}\ar[ddr]^1
&\ar@{}[ddl]|(.6){}="1"|(.75){}="2"
\ar@{=>}"1";"2"^(.3){\eta^{A'\!,B'\!} }
\\&&&&
\\
&F(A)\tm F(B)\ar[rr]_-{F(a)\times F(b)}
&&F(A')\tm F(B')\ar[r]_{R^{A'\!,B'}}
&F(A'\tm B')}
\end{aligned}
\end{equation}
is called the \emph{mate} of $L^{(a,b)}$. A pseudo-commutative square with opposite sides admitting right adjoints, such as the one of diagram~\eqref{diag:L^(a,b)}, is sometimes said to satisfy the Beck-Chevalley condition if its mate is a natural isomorphism.

\begin{Lemma} \label{lemma:mate}
Let $\cB$ be a category with binary products and $P\colon \cX\to\cB$ be an opfibration, and $F\colon \cB\to\CAT$ its associated pseudofunctor.
Moreover, given a pair of arrows $a\colon A\to A'$ and $b\colon B\to B'$ in $\cB$, suppose $L^{A,B}$ and $L^{A'\!,B'}$ have right adjoints $R^{A,B}$ and $R^{A'\!,B'}$ respectively. Then, for an object $(X,Y)$ of $F(A)\times F(B)$, the component $R_{(X,Y)}^{(a,b)}$ of the mate of $L^{(a,b)}$
is the unique $P$-vertical (dashed) arrow such that the following diagram commutes
\begin{equation}\label{diag:mate2}
\begin{aligned}
\xymatrix@C=10ex{
& a_*(X)\times b_*(Y)
\\
X\times Y\ar[r]_-{\widehat{a\times b}}\ar[ur]^{\hat a\times \hat b}
&(a\times b)_*(X\times Y)\ar@{-->}[u]_{R_{(X,Y)}^{(a,b)}}
}
\end{aligned}
\end{equation}
\end{Lemma}

\begin{proof}
The mate of $L^{(a,b)}$ of diagram \eqref{diag:mate} consists in the following vertical composition of the  three horizontal whiskerings:
\[
(R^{A'\!,B'}\circ F(a)\tm F(b) \circ\epsilon^{A,B})
(R^{A'\!,B'}\circ L^{(a,b)}\circ R^{A,B})
(\eta^{A'\!,B'}\circ F(a\tm b)\circ R^{A,B}).
\]
In order to determine its component at an object $(X,Y)$ of $F(A)\times F(B)$,
we can calculate: first
\[
(\eta^{A'\!,B'}\circ F(a\tm b)\circ R^{A,B})_{(X,Y)}=\eta^{A'\!,B'}_{F(a\tm b)\circ R^{A,B}(X,Y)}=\eta^{A'\!,B'}_{(a\tm b)_*(X\tm Y)}=\langle \hat\pi_1,\hat\pi_2 \rangle,
\]
where the last equality follows from the definition of the unit of the adjunction given in  the proof of Lemma~\ref{lemma:R}, then
\[
(R^{A'\!,B'}\circ L^{(a,b)}\circ R^{A,B})_{(X,Y)}=R^{A'\!,B'}\big( L^{(a,b)}_{X\tm Y}\big)=L^{(a,b)}_{X\tm Y,1}\times L^{(a,b)}_{X\tm Y,2}
\]
and last:
\begin{align*}
(R^{A'\!,B'}\circ F(a) \times F(b) & \circ\epsilon^{A,B})_{(X,Y)} =(R^{A'\!,B'}\circ F(a)\tm F(b)) (\epsilon_1,\epsilon_2) = \\
&=R^{A'\!,B'}(a_*(\epsilon_1),b_*(\epsilon_2)) =a_*(\epsilon_1)\times b_*(\epsilon_2),
\end{align*}
where $(\epsilon_1,\epsilon_2)=\epsilon^{A,B}_{(X,Y)}$ is the component of the counit of the adjunction.

Hence we can conclude:
\[
R^{(a,b)}_{(X,Y)}=\big( a_*(\epsilon_1)\times b_*(\epsilon_2)\big)\circ \big(L^{(a,b)}_{X\tm Y,1}\times L^{(a,b)}_{X\tm Y,2} \big)\circ \langle \hat\pi_1,\hat\pi_2 \rangle\,.
\]
Now, let us consider the following diagram.
\[
\xymatrix@C=4ex{
&a_*(X)\times b_*(Y)\ar[r]^-{\pi_2}
&b_*(Y)
\\
&a_*\pi_{1*}(X\tm Y)\times b_*\pi_{2*}(X\tm Y)\ar[r]^-{\pi_2}\ar[u]_{a_*(\epsilon_1)\times b_*(\epsilon_2)}
&b_*\pi_{2*}(X\tm Y)\ar[u]_{b_*(\epsilon_2)}
\\
&\pi_{1*}(a\tm b)_{*}(X\tm Y)\times \pi_{2*}(a\tm b)_{*}(X\tm Y)\ar[r]^-{\pi_2}\ar[u]_{L^{(a,b)}_{X\tm Y,1}\times L^{(a,b)}_{X\tm Y,2}}
&\pi_{2*}(a\tm b)_{*}(X\tm Y)\ar[u]_{ L^{(a,b)}_{X\tm Y,2}}
\\
X\times Y\ar@{.}[d]\ar[r]_-{\widehat{a\times b}}\ar@/^4ex/[ur]_-{\ \ \langle\hat\pi_1\widehat{a\tm b},\hat\pi_2\widehat{a\tm b}\rangle}
\ar@/^7ex/[uur]^{\langle\hat a\hat\pi_1  ,\hat b\hat \pi_2  \rangle}
\ar@/^12ex/[uuur]^{\hat a\times\hat b}
&(a\times b)_*(X\times Y)\ar@{.}[d]\ar[u]_{\langle\hat\pi_1,\hat\pi_2\rangle}\ar[r]_-{\hat\pi_2}
&\pi_{2*}(a\tm b)_{*}(X\tm Y)\ar@{.}[d]\ar[u]_1
\\
A\times B\ar[r]_-{a\times b}&A'\times B'\ar[r]_{\pi_2}
&B'
}
\]
The central vertical composite is precisely the component $R^{(a,b)}_{(X,Y)}$ computed above, therefore, since $\widehat{a\times b}$ is cocartesian, the proof will be completed after showing that all triangles on the left are commutative.
This is easily done by composing such triangles with the product projections (for the sake of clarity, only the second projections are shown in the diagram).
\end{proof}

\begin{Proposition} \label{prop:R}
Let $\cB$ be a category with binary products and $P\colon \cX\to\cB$  be an opfibration, and $F\colon \cB\to\CAT$ its associated pseudofunctor. Then the following statements are equivalent.
\begin{itemize}
\item[i.] The pseudonatural transformation $L$ of Lemma~\ref{lemma:oplax_product_2} has a right adjoint $R$ in the hom-2-category $\mathbf{CAT}(\cB\times\cB,\CAT)$.
\item[ii.] For every pair of objects $A,B$ in $\cB$, $L^{A,B}$ has right adjoint $R^{A,B}$, and for every pair of arrows $a,b$ in  $\cB$, $L^{(a,b)}$ satisfies the Beck-Chevalley condition.
\item[iii.] For every pair of objects $A,B$ in $\cB$, $L^{A,B}$ has right adjoint $R^{A,B}$, and for every pair of arrows $a,b$ as in Lemma~\ref{lemma:mate},  $R_{(X,Y)}^{(a,b)}$ is an isomorphism.
\item[iv.]  The category $\cX$ has binary products, $P$ strictly preserves them and products of cocartesian arrows are cocartesian.
\end{itemize}
\end{Proposition}

\begin{proof}
$(i. \Leftrightarrow ii.)\ $ This is a general fact concerning fibred adjunctions; see for instance  \cite[Proposition 8.4.2]{Borceux1994}.

$(ii. \Leftrightarrow iii.)\ $ This follows directly from Lemma~\ref{lemma:mate}.

$(iii. \Leftrightarrow iv.)\ $ Let us notice that the product of two cocartesian arrows $\hat a,\hat b$ in $\cX$ is cocartesian if, and only if, the comparison $R_{X,Y}^{(a,b)}$ of diagram \eqref{diag:mate2} is an isomorphism. Then the result is a consequence of Corollary~\ref{corollary:R}.
\end{proof}

\begin{Theorem} \label{theorem:adj_make_prod}
Let $\cB$ be a category with finite products and $P\colon \cX\to\cB$ be an opfibration, and $F\colon \cB\to\CAT$ its associated pseudofunctor. Then the following statements are equivalent.
\begin{itemize}
\item[i.] The category $\cX$ has finite products and
\[
P\colon (\cX,\times, I_\cX)\longrightarrow (\cB,\times, I_\cB)
\]
is a monoidal opfibration, i.e.\ $P$ is a \emph{cartesian monoidal opfibration}.
\item[ii.] The pseudonatural transformation $L^1$ of Lemma~\ref{lemma:oplax_unit} has a right adjoint $R^1$ in the hom-2-category $\mathbf{CAT}(\bI,\CAT)$ and the pseudonatural transformation $L$ of Lemma~\ref{lemma:oplax_product_2} has a right adjoint $R$ in the hom-2-category $\mathbf{CAT}(\cB\times\cB,\CAT)$.
\item[iii.] $F$ is endowed with a lax symmetric monoidal structure
\[
(F,R,R^1,\omega',\zeta',\xi',\theta')\colon (\cB,\times, I_\cB)\longrightarrow (\CAT,\times, \cI)
\]
where, for $X,Y$ in $\cX$, $R^{P(X),P(Y)}(X,Y)=X\times Y$ and $R^1(\star)=I_\cX$.
\end{itemize}
\end{Theorem}

\begin{proof}
$(i. \Leftrightarrow ii.)\ $ By Proposition~\ref{prop:Prop_3.2_MV}, we can invoke the double implications $iv.\Leftrightarrow ii.$ of Proposition~\ref{prop:R} and $i.\Leftrightarrow ii.$ of Lemma~\ref{lemma:R1}.

$(i. \Leftrightarrow iii.)\ $ By the symmetric version of the 2-equivalence of \cite[Theorem~3.13]{MV2020},  condition $i.$ on $P$ is equivalent to $F$ being a symmetric monoidal \emph{opindexed category}, and in turn, by \cite[Proposition~3.7]{MV2020}, a lax symmetric monoidal pseudofunctor.

For the reader's convenience, in the following we describe how to obtain the pseudonatural transformations $\omega',\zeta',\xi,\theta'$.

As far as $\omega'$ is concerned, we can consider the pasting diagram
\[
\xymatrix@R=9ex@C=4ex{
(F(A)\tm F(B))\tm F(C)\ar[d]_{R^{A,B}\tm1}\ar[ddrr]^1
&&\ar@{}[ddll]|(.55){}="u4"|(.65){}="u3"|(.8){}="u2"|(.9){}="u1"
\ar@{=>}"u1";"u2"^{\epsilon^{A\tm B,C}}
\ar@{=>}"u3";"u4"_{\epsilon^{A,B}\tm1}
\\
(F(A\tm B))\tm F(C)\ar[d]_{R^{A\tm B,C}}\ar[dr]^1
\\
F((A\tm B)\tm C)\ar[r]_-{L^{A\tm B,C}}\ar[d]_{F(\alpha_{A,B,C})}
&(F(A\tm B))\tm F(C)\ar[r]_-{L^{A,B}\tm1}
&(F(A)\tm F(B))\tm F(C)\ar[d]^{\alpha_{F(A),F(B),F(C)}}
\\
\ar@{}[urr]|(.4){}="1"|(.6){}="2"
F(A\tm (B\tm C))\ar[r]^-{L^{A\tm B,C}}\ar[ddrr]_1
&F(A)\tm F(B\tm C)\ar[r]^-{L^{A,B}\tm1}\ar[dr]_1
&F(A)\tm (F(B)\tm F(C))\ar[d]^{1\tm R^{B,C}}
\\
&&F(A)\tm F(B\tm C)\ar[d]^{R^{A,B\tm C}}
\\
\ar@{}[uurr]|(.55){}="d1"|(.65){}="d2"|(.8){}="d3"|(.9){}="d4"
\ar@{=>}"d1";"d2"^{1\tm \eta^{B,C}}
\ar@{=>}"d3";"d4"_{\eta^{A,B\tm C}}
&&F(A\tm (B\tm C))
\ar@{=>}"1";"2"_{\ \ \ \omega_{A,B,C}}^{\cong\ \ }
}
\]
which is nothing but the mate of $\omega_{A,B,C}$.
It is possible to prove that, for $X,Y,Z$ in $\cX$ such that $P(X)=A$, $P(Y)=B$ and $P(Z)=C$, the component $(\omega'_{A,B,C})_{X,Y,Z}$ is the unique $P$-vertical (dashed) arrow that makes the following diagram commute
\[
\xymatrix@C=10ex{
&X\tm (Y\tm Z)
\\
(X\tm Y)\tm Z\ar[r]_-{\hat\alpha_{A,B,C}}\ar[ur]^-{\alpha_{X,Y,Z}}
&\alpha_{A,B,C\,*}((X\tm Y)\tm Z)\ar@{-->}[u]}
\]
It is an isomorphism since both $\hat\alpha_{A,B,C}$ and $\alpha_{X,Y,Z}$ are.

As far as $\zeta'$ is concerned (the construction of $\xi'$ is similar), we can consider the pasting diagram
\[
\xymatrix@R=10ex{
F(A)\tm \cI\ar[d]_{1\tm R^1}\ar[ddrr]^1
&&\ar@{}[ddll]|(.55){}="u4"|(.65){}="u3"|(.8){}="u2"|(.9){}="u1"
\ar@{=>}"u1";"u2"^{\epsilon^{A,I_\cB}}
\ar@{=>}"u3";"u4"_{1\tm \epsilon^1}
\\
F(A)\tm F(I_\cB)\ar[d]_{R^{A,I_\cB}}\ar[dr]^1
\\
F(A\tm I_\cB)\ar[r]_-{L^{A\tm I_\cB}}\ar[d]_{F(\pi_1)}
&F(A)\tm F(I_\cB)\ar[r]_-{1\tm L^1}
&F(A)\tm \cI\ar[d]^{\pi_1}
\\
F(A)\ar[rr]_1
\ar@{}[urr]|(.4){}="1"|(.6){}="2"
\ar@{=>}"1";"2"_{\zeta_A}^{\cong}
&&F(A)
}
\]
which is nothing but the mate of $\zeta_A$.
It is possible to prove that, for $X$ in $\cX$ such that $P(X)=A$, the component $(\zeta'_{A})_{X}$ is the unique $P$-vertical (dashed) arrow that makes the following diagram commute:
\[
\xymatrix@C=10ex{
&X
\\
X\tm I_\cX\ar[r]_-{\hat\pi_1}\ar[ur]^-{\pi_1}
&\pi_{1*}(X\tm I_\cX)\ar@{-->}[u]}
\]
It is an isomorphism since both $\hat\pi_1$ and $\pi_1$ are.
Finally, for all $A,B$ in $\cB$, $\theta'_{A,B}$ is obtained as the mate of $\theta_{A,B}$, which is an isomorphism for the same argument as above.
\end{proof}

\begin{Remark} \label{rmk:sym}
Notice that, under the hypothesis of point $iii.$ of Theorem \ref{theorem:adj_make_prod}, $F$ sends commutative monoids in $(\cB,\times, I_\cB)$ to symmetric monoidal categories. The following diagram displays how the braiding $\tau_{X,Y}$ is constructed for given objects $X$ and $Y$ over a commutative monoid $A$.
\[
\xymatrix@C=5ex{
X \times Y \ar@{.}[ddddd] \ar@/_{8ex}/[ddrddr]_{1_{X\times Y}} \ar[ddr]_{\sigma} \ar[dr]^{\widehat{\sigma}} \ar[rrr]^{\widehat{m}} & & & X \otimes Y \ar[d]_{\phi^{-1}_{\sigma,m}}^{\sim} \ar@/^{10ex}/[dd]^{\tau_{X,Y}} \\
& \sigma_*(X \times Y) \ar[d]_{\sim}^{\theta'_{X,Y}} \ar[rr]^{\widehat{m}} & & m_*\sigma_*(X \times Y) \ar[d]_{m_*(\theta')}^{\sim} \\
& Y \times X \ar@{.}[ddd] \ar[dr]^{\widehat{\sigma}} \ar[ddr]_{\sigma} \ar[rr]^{\widehat{m}} & & Y \otimes X \ar[d]_{\phi^{-1}_{\sigma,m}}^{\sim} \ar@/^{10ex}/[dd]^{\tau_{Y,X}} \\
& & \sigma_*(Y \times X) \ar[d]_{\sim}^{\theta'_{Y,X}} \ar[r]^{\widehat{m}} & m_*\sigma_*(Y \times X) \ar[d] \ar[d]_{m_*(\theta')}^{\sim} \\
& & X \times Y \ar@{.}[d] \ar[r]^{\widehat{m}} & X \otimes Y \ar@{.}[d] \\
A \times A \ar[r]^{\sigma} \ar@/_{6ex}/[rrr]_m & A \times A \ar[r]^{\sigma} \ar@/_{3ex}/[rr]_m & A \times A \ar[r]^m & A
}
\]
\end{Remark}

\section{When groups are sent to 2-groups} \label{sec:groups}

\begin{Lemma} \label{lemma:unit_terminal}
Let the equivalent conditions of Lemma~\ref{lemma:R1} hold. The following statements are equivalent.
\begin{itemize}
\item[i.] The unit of the adjunction $\eta^1\colon 1_{F(I_\cB)}\Rightarrow R^1\circ L^1$ is a natural isomorphism.
\item[ii.] For all $X$ in $F(I_{\cB})$, the terminal map $\tau_X\colon X\to I_{\cX}$ is $P$-cocartesian.	
\item[iii.] For all $X$ in $\cX$, the terminal map $\tau_X\colon X\to I_{\cX}$ is $P$-cocartesian.
\end{itemize}
\end{Lemma}
\begin{proof}
$(i. \Rightarrow ii.)\ $ This is obvious, since, as we saw in Lemma~\ref{lemma:R1}, for $X$ in $F(I_\cB)$, $\eta^1_X=\tau^{\cX}_X$.

$(ii. \Rightarrow iii.)\ $ For $X$ in $\cX$, the terminal map $\tau_X$ is the composite of a cocartesian lifting followed by a terminal map in $F(I_\cB)$, which is $P$-cocartesian by $ii.$

$(iii. \Rightarrow i.)\ $ For all $X$ in $F(I_\cB)$, the $P$-vertical  map $\eta^1_X=\tau_X$ is  $P$-cocartesian, hence an isomorphism.
\end{proof}

\begin{Lemma} \label{lemma:unit_product}
Let $P\colon \cX\to\cB$ be a cartesian monoidal opfibration, and $F\colon \cB\to\CAT$ its associated pseudofunctor. The following statements are equivalent.
\begin{itemize}
\item[i.] For all $A$ in $\cB$, the unit of the adjunction $\eta^{A,A}\colon 1_{F(A\times A)}\Rightarrow R^{A,A}\circ L^{A,A}$ is a natural isomorphism.
\item[ii.] For all $X$ in $\cX$,  the diagonal map $\Delta_X\colon X\to X\times X$ is $P$-cocartesian.	
\end{itemize}
\end{Lemma}

\begin{proof}
$(i. \Rightarrow ii.)\ $ Given an object $X$ in $\cX$, let $A=P(X)$.
\[
\xymatrix@C=10ex{
&X\tm X
\ar[r]^{\pi_i}
&X
\\
X\ar@{.}[d]\ar[r]_-{\hat\Delta_A}\ar[ur]^{\Delta_X}
&\Delta_{A*}(X)\ar[r]_{\hat\pi_i}\ar@{.}[d]\ar[u]^{\phi}
&\pi_{i*}\Delta_{A*}(X)\ar@{.}[d]\ar[u]_{\phi_i}
\\
A\ar[r]^{\Delta_A}\ar@/_4ex/[rr]_{1_{A}}
&A\tm A\ar[r]^-{\pi_i}
&A
}
\]
Since  $\hat\pi_i\circ\hat\Delta_A$ is cocartesian (for $i=1,2$), then there exists a $P$-vertical $\phi_i$ (necessarily  isomorphism) such that   $\phi_i\circ \hat\pi_i\circ\hat\Delta_A=\pi_i\circ \Delta_X=1_X$. Similarly, because  $\hat\Delta_A$ is cocartesian, there exists a $P$-vertical $\phi$ such that $\phi\circ \hat\Delta_A=\Delta_X$ and $\pi_i\circ \phi=\phi_i\circ \hat\pi_i$. Then, $\phi=\phi_1\tm\phi_2\circ\langle\hat\pi_1,\hat\pi_2 \rangle=\phi_1\tm\phi_2\circ \eta^{A,A}_{\Delta_{A*}(X)}$, hence an isomorphism. Consequently, $\Delta_X$ is $P$-cocartesian.

$(ii. \Rightarrow i.)\ $
Let $Z$ be an object of $F(A\times A)$. Then, since both $\Delta_Z$ and $\hat\pi_1\tm\,\hat\pi_2$ are cocartesian, also their composite $\langle \hat\pi_1,\hat\pi_2 \rangle=\eta_{Z}^{A,A}$ is. But the latter is also $P$-vertical, hence an isomorphism.
\[
\xymatrix@C=10ex{
Z\ar[r]_{\Delta_Z}\ar@/^4ex/[rr]^{\eta_{Z}^{A,A}}\ar@{.}[d]
&Z\tm Z\ar[r]_-{\hat\pi_1\tm\,\hat\pi_2}\ar@{.}[d]
&\pi_{1*}(Z)\tm \pi_{2*}(Z)\ar@{.}[d]
\\
A^2\ar[r]^{\Delta_{A^2}}\ar@/_4ex/[rr]_{1_{A^2}}
&A^2\tm A^2\ar[r]^-{\pi_1\tm\,\pi_2}
&A^2
}
\]
\end{proof}

Let $\cB$ be a category with finite products. Recall that a group object in $\cB$ is a monoid object $(A, m, e)$ endowed with an arrow $i\colon A\to A$ such that the following diagram commutes:
\begin{equation} \label{diag:group_inverses}
\begin{aligned}
\xymatrix@C=10ex{
A\ar[d]_{\tau_A}\ar[r]^-{\langle 1, i\rangle}
&A\times A\ar[d]^{m}
\\
I_\cB\ar[r]_{e}
&A}
\end{aligned}
\end{equation}

\begin{Definition} \label{def:gps}
A lax monoidal pseudofunctor $F\colon(\cB,\times,I_B)\to(\CAT,\times,\cI)$ is called \emph{groupal} when it sends internal groups in $\cB$ to 2-groups, namely, it determines the pseudofunctor $\CG(F)$ that makes the diagram
\begin{equation} \label{diag:CG}
\begin{aligned}
\xymatrix@C=16ex{
\Gp(\cB)\ar[d]_U\ar[r]^{\CG(F)}
&\CG\ar[d]^U
\\
\cB\ar[r]_{F}
&\CAT
}
\end{aligned}
\end{equation}
commute, where the vertical arrows are the forgetful functors.
\end{Definition}

The following theorem provides a full characterization of groupal pseudo\-functors.

\begin{Theorem} \label{thm:Gp(B)}
Let $P\colon \cX\to\cB$ be a cartesian monoidal opfibration, and $F\colon \cB\to\CAT$ its associated pseudofunctor. The following statements are equivalent.
\begin{itemize}
\item[i.] The equivalent conditions of Lemma~\ref{lemma:unit_terminal} hold, and condition $i.$ of Lemma \ref{lemma:unit_product} holds for every object $A$ in $\cB$ endowed with an internal group structure $(A,m,e,i)$.
\item[ii.] The lax monoidal pseudofunctor $F$ is groupal.
\end{itemize}
\end{Theorem}

\begin{proof}
$(i.\Rightarrow ii.)\ $ By the equivalent conditions of Theorem~\ref{theorem:adj_make_prod}, $F$ is lax monoidal with respect to the cartesian monoidal structure of $\cB$. As we recalled at the end of Section~\ref{sec:preliminaries}, such an $F$ endows $F(A)$ with a  tensor product given by $\otimes=F(m)\circ R^{A,A}$, and a unit object $E=(F(e)\circ R^1)(\star)$.

We will show that the functor $(\ \, )^*=F(i)\colon F(A)\to F(A)$ provides pseudoinverses for the monoidal category $(F(A),\otimes,E)$. Let us consider the following diagram:
\[
\xymatrix@C=9ex@R=9ex{
F(A)\ar[rr]^-{\Delta_{F(A)}}\ar[dr]_{F(\Delta_A)}\ar[d]_{F(\tau_A)}
&\ar@{}[d]|{\cong}
&F(A)\tm F(A)\ar@/^2ex/[r]^-{1\tm F(i)}\ar@/_2ex/[r]_-{F(1)\tm F(i)}\ar@{}[d]|(.5){\cong}|(.4){L^{(1,i)}}\ar@{}[r]|{\cong}
&F(A)\tm F(A)\ar[d]^{R^{A,A}}\ar@{}|{}="y"
\\
F(I_\cB)\ar[d]_{L^1}\ar[dr]^-1
&F(A\tm A)\ar[ur]_-{L^{A,A}}\ar[r]_-{F(1\tm i)}\ar@{}[dr]|{\cong}
&F(A\tm A)\ar[ur]^(.3){L^{A,A}}\ar[r]_-1="x"
&F(A\tm A)\ar[d]^{F(m)}
\\
\cI\ar[r]_{R^1}\ar@{}[ur]|(.2){}="e1"|(.4){}="e2"
&F(I_\cB)\ar[rr]_{F(e)}
&&F(A)
\ar@{=>}"e1";"e2"_{(\eta^1)^{-1}}
\ar@{}"y";"x"|(.35){}="ee1"|(.65){}="ee2"
\ar@{=>}"ee2";"ee1"_{\eta^{A,A}}
}
\]
where, since $L^{A,A}=\langle F(\pi_1),F(\pi_2) \rangle$, all unlabelled natural isomorphisms are given by the pseudofunctoriality of $F$.
Then, since $\eta^1$ and $\eta^{A,A}$ are natural isomorphisms by hypothesis, the pasting of the diagram gives a natural isomorphism $\gamma$ with components $\gamma_X\colon E \rightarrow  X\otimes X^*$, for  $X$ an object of $F(A)$.
A similar construction, which uses $L^{(i,1)}$ instead of $L^{(1,i)}$, leads to the definition of a natural isomorphism $\delta'$ with components $\delta'_X\colon X^*\otimes X\rightarrow E$. Now, it is a standard well-known fact that, given $\gamma_X,\delta'_X$ as above, it is possible to define  a natural isomorphism $\delta$ such that the pair $(\gamma_X,\delta_X)$ satisfies the usual triangle identities:
\[
(1_X\otimes \delta_X)\circ\alpha_{X,X^*,X}\circ(\gamma_X\otimes 1_X)=\rho_X^{-1}\circ\lambda_X
\]
\[
(\delta_X\otimes1_{X^*})\circ\alpha_{X^*,X,X^*}\circ (1_{X^*}\otimes \gamma_X)=\lambda^{-1}_{X^*}\circ \rho_{X^*}
\]
(see for instance the proof of \cite[Theorem~5.1]{BL2004}).
Moreover, every arrow $f\colon X\to Y$ in $F(A)$ is an isomorphism, with inverse given below:
\[
\xymatrix@C=8ex{
Y\ar[r]^-{\lambda_Y^{-1}}
&E\otimes Y\ar[r]^-{\gamma_X\otimes 1}
& (X\otimes X^*)\otimes Y\ar[r]^-{(1\otimes f^*)\otimes 1}
&(X\otimes Y^*)\otimes Y\ar[dll]^-{\ \ \alpha_{X,Y^*,Y}}
\\
&X\otimes (Y^*\otimes Y)\ar[r]_-{1\otimes \delta_Y}
&X\otimes E\ar[r]_-{\rho_X}
&X
}
\]
(see for instance the proof of \cite[Proposition~7.2]{BL2004}).
Hence $(F(A), \otimes, E, (\ )^*)$ is a 2-group. Moreover, since for a group morphism $f\colon A\to B$, $F(f)$ is monoidal, then $F$ determines a pseudofunctor $\CG(F)$ such that diagram \eqref{diag:CG} commutes.

$(ii.\Rightarrow i.)\ $ Since $\cB$ has finite products, the category of groups in $\cB$ has finite products as well.
For the terminal group  $I_\cB$, the monoidal category $F(I_\cB)$ is actually  a 2-group, hence a groupoid; as a consequence, $\eta^1$ must be an isomorphism.
Similarly, for a group $A$ in $\cB$, $F(A\times A)$ is also a 2-group, hence a groupoid; as a consequence, $\eta^{A,A}$ must be a natural isomorphism.
\end{proof}

\begin{Proposition} \label{prop:abelian}
A groupal pseudofunctor $F\colon(\cB,\times,I_B)\to(\CAT,\times,\cI)$ lifts to a pseudofunctor from  $\Ab(\cB)$ to $\SCG$.
\end{Proposition}

\begin{proof}
Just observe that abelian groups are commutative monoids, so that the lax monoidal pseudofunctor $F$ sends them to 2-groups, which are symmetric monoidal thanks to Theorem \ref{theorem:adj_make_prod} (see also Remark \ref{rmk:sym}), namely to symmetric 2-groups.
\end{proof}

\begin{Remark}
In \cite[Theorem 9]{Bourn99} Bourn proved that, for an opfibration $P\colon \cX\to\cB$ which preserves binary products and terminal object and such that cocartesian maps are stable under product, if terminal maps and diagonals are cocartesian, then an abelian group structure on an object of $\cB$ induces a closed monoidal structure on its fibre.  This result can be obtained as a consequence of our Proposition~\ref{prop:abelian}. Actually, it suffices to assume the hypotheses concerning terminal and diagonal maps only for the objects of \cX\ whose image supports a group structure. Moreover, from our perspective, one also sees that the fibre over any group object is a groupoid.
\end{Remark}

\begin{Corollary} \label{cor:groupoidal_fibres}
Let $P\colon \cX\to\cB$ be a cartesian monoidal opfibration whose fibres are groupoids, and $F\colon \cB\to\CAT$ its associated pseudofunctor. Then $F$ lifts to pseudofunctors from $\Gp(\cB)$  to $\CG$ and from $\Ab(\cB)$ to $\SCG$.
\end{Corollary}

\begin{proof}
Since the fibres of $P$ are groupoids, the units of the adjuntions of Lemma~\ref{lemma:unit_terminal} and Lemma~\ref{lemma:unit_product} are isomorphisms, and we can apply Theorem~\ref{thm:Gp(B)} and Proposition \ref{prop:abelian}.
\end{proof}

\begin{Theorem} \label{thm:additive}
Let the category $\cB$ be  additive, $P\colon \cX\to\cB$ be a cartesian monoidal opfibration, and $F\colon \cB\to\CAT$ its associated pseudofunctor.
Suppose moreover that, for every $X$ in $\cX$, the terminal map $\tau_X$ and the diagonal map $\Delta_X$ are $P$-cocartesian. Then, $F$ factors through $\SCG$.
\end{Theorem}

\begin{proof}
When $\cB$ is additive, every object of $\cB$ has a unique abelian group structure, and all morphisms of $\cB$ preserve such structures. As a consequence, the forgetful functor $U\colon \Ab(\cB)\to \cB$ is an isomorphism, so that the conditions on terminal and diagonal maps apply to every object in \cX, and we can apply Proposition~\ref{prop:abelian}.
\end{proof}

From Lemma \ref{lemma:unit_terminal}, Lemma \ref{lemma:unit_product} and Theorem \ref{thm:additive}, we finally get one of the main results of this paper.

\begin{Theorem} \label{thm:final}
Let the category $\cB$ be  additive, $P\colon \cX\to\cB$ be a cartesian monoidal opfibration, and $F\colon \cB\to\CAT$ its associated pseudofunctor. Then, $F$ factors through $\SCG$ if and only if $P$ has groupoidal fibres.
\end{Theorem}

\section{Examples}

\subsection{Abelian extensions of groups} \label{sec:H^2}

Let us fix a group $C$ and consider the category $\ABEXT(C)$ of group extensions of $C$ with abelian kernel, and the abelian category $\Mod(C)$ of $C$-modules.
Moreover, let us consider the functor $P\colon \ABEXT(C)\to \Mod(C)$ which associates with every abelian extension the induced action of $C$ on the kernel:
\begin{equation} \label{diag:opext_to_mod}
P\colon\quad
\begin{aligned}
\xymatrix{B\ar[r]^{k}\ar[d]_{\phi}&E\ar[r]^{f} \ar[d]^{\psi}&C\ar[d]^{1}
\\
B'\ar[r]_{k'}&E'\ar[r]_{f'} &C}
\end{aligned}
\qquad\mapsto\qquad
\begin{aligned}
\xymatrix{C\times B\ar[r]^-{\xi}\ar[d]_{1\times \phi}&B\ar[d]^{\phi}
\\
C\times B'\ar[r]_-{\xi'}&B'}
\end{aligned}
\end{equation}
In fact, $P$ is an opfibration. With notation as above, a cocartesian lifting of $\phi$ at $E=(f,k)$ is obtained by the so-called \emph{pushforward construction}:
\begin{equation*}
\begin{aligned}
\xymatrix{B\ar[r]^{k}\ar[d]_{\phi}&E\ar[r]^{f} \ar[d] &C\ar[d]^{1}
\\
B'\ar[r]&\phi_{*}E\ar[r]&C}
\end{aligned}
\end{equation*}
where $\phi_{*}E$ is a quotient of the semidirect product with the action of $E$ on $B'$ induced by $f$ (this is classical, see for instance \cite{Homology}).
It is easy to check that $\ABEXT(C)$ has terminal object and binary products, given by pullbacks over $C$; moreover $P$ preserves them, and  products of cocartesian maps are still cocartesian. Therefore we can apply Proposition \ref{prop:R} and Theorem \ref{theorem:adj_make_prod} in order to describe the lax monoidal structure of the pseudofunctor associated with $P$:
\[
\OPEXT(C,-)\colon \Mod(C)\to \CAT.
\]
Our notation is consistent with literature (e.g.\ \cite{Vitale03}), since the fibre of $P$ over a $C$-module $(B,\xi)$ is the groupoid $\OPEXT(C,B,\xi)$  of abelian extensions of $C$, with kernel $B$ and induced action $\xi$.
We can apply Theorem \ref{thm:final}, so that each fibre of $P$ gets a symmetric 2-group structure, and change of base functors are symmetric monoidal. Actually, the monoidal structure is  obtained by \eqref{diag:F(M)tensor} and \eqref{diag:F(M)unit} as follows: given a pair of extensions $(E,E')\in \OPEXT(C,B,\xi)\times \OPEXT(C,B,\xi)$
\begin{itemize}
\item the right adjoint $R^{(B,\xi),(B,\xi)}$ yields the product of $E$ and $E' $ over $C$,
\item the codomain of the cocartesian lifting at $R^{(B,\xi),(B,\xi)}$ of the group operation $m\colon B\times B\to B$ gives the tensor product $E\otimes E'$.
\[
\xymatrix@C=16ex{
B\ar[r]&E\otimes E'\ar[r] &C\\
B\times B\ar[d]_{1} \ar[u]^m \ar[r] &R^{(B,\xi),(B,\xi)}(E,E')\ar[u]\ar[d]\ar[r] \ar[r]\ar@{}[dr]|{\text{pullback}}\ar@{}[ul]|{\text{push forward}}& C\ar[d]^{\Delta_C}\ar[u]_{1}\\
B\times B\ar[r]_{k\times k'} & E\times E'\ar[r]_{f\times f'}&C\times C}
\]
\end{itemize}

Once one considers connected components, the tensor operation induces precisely the Baer sum of isomorphism classes of extensions, which makes $\pi_0(\OPEXT(C,B,\xi))$ a group (isomorphic to $H^2(C,B,\xi)$). On the other hand, one can compute the abelian group $\pi_1(\OPEXT(C,B,\xi))$ of the automorphisms of the identity object:
\[
\xymatrix{
B\ar[r]&B\rtimes_{\xi}C\ar[r]&C}\,.
\]
According to Proposition IV.2.1 in \cite{Homology}, this can be interpreted as the group of 1-cocycles $Z^1(C,B,\xi)$, as already noticed in \cite{categorical_ot}.
\begin{Remark}
Section \ref{sec:H^2} can be easily adapted to the case of semi-abelian categories (see \cite{pf} and also Section 4.2 of \cite{categorical_ot}, where the general case of Bourn's direction functor is considered). The analogous result holds for associative unitary algebras (see Section 4.3 of \cite{categorical_ot}).
\end{Remark}

\subsection{From left monoid actions to torsors} \label{sec:torsors}

Let us briefly recall the category $\mathsf{LMSet}$ of monoid left actions. An object of  $\mathsf{LMSet}$ is a set $X$ endowed with a left action $\phi\colon M\times X\to X$ of a monoid $M$, where  $m\in M$ operates on  $x\in X$ by the law $m\cdot x = \phi(m,x)$.
Given two such objects $(M,X,\phi)$ and $(N,Y,\psi)$, an arrow between them is a pair $(f,f_0)$ where  $f\colon M\to N$ is a monoid homomorphism and $f_0 \colon X\to Y$ is a function  such that the following diagram commutes in $\Set$:
\[
\xymatrix{
M\times X\ar[r]^-{\phi}\ar[d]_{f\times f_0}
&X\ar[d]^{f_0}
\\
N\times Y\ar[r]_-{\psi}
&Y
}
\]
Such a pair $(f,f_0)$ is called \emph{equivariant.}

In the following proposition, we will use a generalization of the classical tensor product of (bi)modules (see for instance \cite[VII.4, Exercise 6]{Categories}). For a right $M$-set X and a left $M$-set $Y$, their \emph{contracted product} $X\overset{M}{\wedge}{Y}$ is the quotient of $X\times Y$ determined by the equivalence relation generated by $(x\cdot m,y)\sim(x,m\cdot y)$. The term contracted product and the wedge notation are borrowed from \cite{Breen}, where the special case of bitorsors is dealt with.

\begin{Proposition} \label{prop:62}
The forgetful functor $V\colon \mathsf{LMSet}\to \Mon$ is a cartesian\\
monoidal opfibration.
\end{Proposition}

\begin{proof}
We only sketch the proof, which is based on well-known facts and constructions. Let $(M, X,\phi)$ be an object of $ \mathsf{LMSet}$, and $f\colon M\to N$ a homomorphism.  The product in $N$ and precomposition with $f$ make $N$ a right $M$-set. Therefore, we can define $f_!(X)$ as the contracted product $N\overset{M}{\wedge} X$.  In fact, $f_!(X)$ is a left $N$-set, with  action defined by $\overline n\cdot[n,x]=[\overline n n, x]$, square brackets denoting equivalence classes. Moreover, the assignment $x\to[1,x]$ is well defined and yields a function $\hat f\colon X\to f_!(X)$ such that $(f,\hat f)$ is an equivariant pair. Easy calculations show that $(f,\hat f)$ is $V$-cocartesian over $f$, and this makes $V$ an opfibration.

Concerning the  monoidal structure of the opfibration, it is clear that $V$ preserves products. Let us sketch how the product of cocartesian arrows is itself cocartesian. For objects $(M_i,X_i,\phi_i)$ and homomorphisms $f_i\colon M_i\to N_i$ as above, with $i=1,2$, the identifications
\[
[(n_1,n_2),(m_1\cdot x_1,m_2\cdot x_2)]=[(n_1,n_2), (m_1,m_2)\cdot(x_1,x_2)]=
\]
\[
=[(n_1,n_2)\, \,(f_1\times f_2)(m_1,m_2),(x_1,x_2)]=[ (n_1f_1(m_1)  ,n_2f_2(m_2))  ,(x_1,x_2)]
\]
show that the assignment
\[
([n_1,x_1],[n_2,x_2])\mapsto [(n_1,n_2),(x_1,x_2)]
\]
is well defined.
Moreover, it gives an equivariant isomorphism
\[(f_1)_!(X_1)\times (f_2)_!(X_2)\cong (f_1\times f_2)_!(X_1\times X_2)\,.\]
\end{proof}

Thus, $V$ translates into a lax symmetric  monoidal  pseudofunctor denoted by
\[
\ACT\colon \Mon\to \CAT
\]
with monoidal structure given as in Theorem \ref{theorem:adj_make_prod}.

Let us observe that Theorem \ref{thm:Gp(B)} cannot be applied to $\ACT$ in its full generality. We will soon deal with this matter, but first, let us briefly discuss the monoidal structure of the fibres.

By Proposition \ref{prop:62}, $\ACT$ lifts to a lax symmetric monoidal  pseudofunctor between internal pseudomonoids; indeed, pseudomonoids in the $1$-category of monoids are just  commutative monoids, and then, thanks to Remark \ref{rmk:sym}, we have:
\[
\ACT\colon \CMon\to \SMC.
\]

Let us fix a commutative monoid  $M$.  The tensor product in $\ACT(M)$ is obtained by the construction \eqref{diag:F(M)tensor} as follows: given a pair of left $M$-sets $X$ and $Y$,
\begin{itemize}
\item the right adjoint yields the product of $X$ and $Y$ in $\mathsf{LMSet}$,
\item the codomain of the cocartesian lifting of the monoid operation $p:M\times M\to M$ gives the tensor product $p_!(X\times Y)=X\otimes_M Y$.
\end{itemize}
The following straightforward result shows that it is indeed an instance of the contracted product already mentioned before.

\begin{Proposition} \label{prop:contracted}
The map  $X\overset{M}{\wedge} Y\to X\otimes_M Y$ given by the assignment
\[
[x,y]\mapsto [1,x,y]
\]
is an isomorphism of $M$-sets.
\end{Proposition}

A monoid which supports an internal group structure is in fact an abelian group. In the following statement, for each abelian group $B$, we characterize $B$-torsors among $B$-sets. Recall that a $B$-set $X$ is called a $B$-torsor if it is not empty and the assignment
\begin{equation}\label{eq:trans}
B\times X\to X\times X\qquad (b,x)\mapsto (x,b\cdot x)
\end{equation}
is a bijection. We shall call such an object \emph{abelian torsor} and denote by $\mathsf{ATors}$ the full subcategory of $\mathsf{LMSet}$ with abelian torsors as objects.

\begin{Proposition} \label{prop:64}
Let $B$ be an abelian group. A $B$-set $X$ is a $B$-torsor if and only if it satisfies the following conditions:
\begin{enumerate}
 \item the terminal map $\tau_X\colon X \to I$ is $V$-cocartesian;
 \item the diagonal map $\Delta_X\colon X\to X\times X$ is $V$-cocartesian.
\end{enumerate}
\end{Proposition}

\begin{proof}
Let us consider a $B$-set $X$ and $\tau_B\colon B\to 0$ the terminal map in \Mon.

$\tau_X$ is $V$-cocartesian if and only if $(\tau_B)_{!}(X)=0\overset{B}{\wedge}X\cong I$, and this is equivalent to saying that $X$ is not empty and the function \eqref{eq:trans} is surjective.

Now take $X$ non-empty. $\Delta_X$ is cocartesian if and only if the comparison
\[
\phi\colon (\Delta_B)_{!}(X)=(B\times B)\overset{B}{\wedge} X\to X\times X
\]
is an isomorphism, and one easily computes:\[\phi([b_1,b_2,x])=(b_1\cdot x,b_2\cdot x)\,.\]
Now, since for all $b_1,b_2\in B$ and $x\in X$ $[b_1,b_2,x]=[1,b_2b_1^{-1},b_1\cdot x]$, the map $\psi\colon B\times X\to (B\times B)\overset{B}{\wedge} X$  given by the assignment $(b,x)\mapsto [1,b,x]$ is a bijection. So $\phi$ is an isomorphism if and only if $\psi\cdot\phi\colon B\times X\to X\times X$, $(b,x)\mapsto(x,b\cdot x)$, is bijective.
\end{proof}

The restriction of the forgetful functor $V$ to $\mathsf{ATors}$ yields a monomorphism of monoidal opfibrations
\[
\xymatrix{
\mathsf{ATors\ } \ar@{^{(}->}[r] \ar[d]_{\overline V} & \mathsf{LMSet} \ar[d]^V \\
\Ab\ \ar@{^{(}->}[r] & \Mon
}
\]
since contracted products of (abelian) torsors are still (abelian) torsors. Moreover, by Proposition \ref{prop:64}, $\overline V$ is actually the largest restriction of $V$ with base $\Ab$ with groupoidal fibres. Thanks to Corollary \ref{cor:groupoidal_fibres}, the corresponding pseudo\-functor $\mathsf{TORS}$ factors through symmetric 2-groups:
\[
\xymatrix@C=16ex{
&\SCG\ar[d]^U
\\
\Ab\ar[r]_{\mathsf{TORS}}\ar[ur]
&\CAT
}
\]
In particular, for an abelian group $B$, one obtains a symmetric 2-group structure on $\mathsf{TORS}(B)$. This is a special case of the classical 2-group of bitorsors over a (not necessarily abelian) group $B$, as studied by Breen in \cite{Breen}.


\subsection*{Acknowledgment}

Partial financial support was received from INdAM -- Istituto Nazionale di Alta Matematica ``Francesco Severi'' -- Gruppo Nazionale per le Strutture Algebriche, Geometriche e le loro Applicazioni.


\begin{thebibliography}{99}

\bibitem{BL2004} J. Baez and A. Lauda. Higher-dimensional algebra V: 2-Groups, \emph{Theory Appl.~Categ.} 12 (2004) 423--491.

\bibitem{Borceux1994} F.~Borceux. \emph{Handbook of Categorical Algebra,} vol. 2. \emph{Encyclopedia of Mathematics and its Applications} 50, Cambridge University Press (1994).

\bibitem{Bourn99} D.~Bourn. Baer sums and fibred aspects of Mal'cev operations, \emph{Cah.~Topol. G\'eom.~Diff\'er.~Cat\'eg.} XL (1999) 297--316.

\bibitem{BJ} D.~Bourn and G.~Janelidze. Extensions with abelian kernel in protomodular categories, \emph{Georgian Math.~J.} 11 (2004) 645--654.

\bibitem{Breen} L.~Breen. Bitorseurs et Cohomologie Non Ab\'elienne, in \emph{The Grothendieck Festschrift}, Vol. I, 401--476. Progress in Mathematics, 86. Birkh\"auser, Boston, MA (2007).

\bibitem{pf} A.~S.~Cigoli, S.~Mantovani, and G.~Metere. A push forward construction and the comprehensive factorization for internal crossed modules, \emph{Appl.~Categ.~Structures} 22 (2014) 931--960.

\bibitem{categorical_ot} A.~S.~Cigoli, S.~Mantovani, G.~Metere, and E.~M.~Vitale. Fibred-categorical obstruction theory, \emph{J.~Algebra} 593 (2022) 105--141.

\bibitem{CM16} A.~S.~Cigoli and G.~Metere. Extension theory and the calculus of butterflies, \emph{J.~Algebra} 458 (2016) 87--119.

\bibitem{DS1997} B.~Day and R.~Street. Monoidal Bicategories and Hopf Algebroids, \emph{Adv.~Math.} 129 (1997) 99--157.

\bibitem{GPS1995} R.~Gordon, A.~J.~Power, and R.~Street. Coherence for tricategories, \emph{Mem. Amer. Math. Soc.} 558 (1995).

\bibitem{Categories} S.~Mac Lane. \emph{Categories for the Working Mathematician.} Graduate Texts in Mathematics, Vol. 5, Springer-Verlag, New York, (1971).

\bibitem{Homology} S.~Mac Lane. \emph{Homology.} Die Grundlehren der mathematischen Wissenschaften, Bd. 114. Academic Press, Inc., Publishers, New York; Springer-Verlag, Berlin-G\"ottingen-Heidelberg (1963).

\bibitem{McC2000} P.~McCrudden. Balanced coalgebroids, \emph{Theory Appl.~Categ.} 7 (2000) 71--147.

\bibitem{MV2020} J.~Moeller and C.~Vasilakopoulou. Monoidal Grothendieck Construction, \emph{Theory Appl.~Categ.} 35 (2020) 1159--1207.

\bibitem{Shu2008} M.~Shulman. Framed bicategories and monoidal fibrations, \emph{Theory Appl.~Categ.} 20 (2008) 650--738.

\bibitem{Vitale03} E.~M.~Vitale. On the categorical structure of $H^2$, \emph{J.~Pure Appl.~Algebra} 177 (2003) 303--308.

\end{thebibliography}
\end{document}